\documentclass[12pt,a4paper,reqno]{amsart}
\usepackage{amsmath}
\usepackage{amsfonts}
\usepackage{amssymb}
\usepackage{amsthm}
\usepackage{enumerate}
\usepackage[all]{xy}
\usepackage{graphicx}
\usepackage{fullpage}
\usepackage{mathrsfs}
\usepackage{epstopdf}

\theoremstyle{plain}

  \newtheorem{proposition}[]{Proposition}
  \newtheorem{lemma}[]{Lemma}
  \newtheorem{theorem}[]{Theorem}
  \newtheorem{corollary}[]{Corollary}
  
  \newtheorem{remark}[]{Remark}

\theoremstyle{definition}

\title{Critical first-passage percolation starting on the boundary}
\author{Jianping Jiang}
\address{NYU-ECNU Institute of Mathematical Sciences at NYU Shanghai, 3663 Zhongshan
Road North, Shanghai 200062, China.}
\email{jjiang@nyu.edu}
\author{Chang-Long Yao}
\address{Academy of Mathematics and Systems Science, CAS, Beijing, China.}
\email{deducemath@126.com}

\begin{document}
\begin{abstract}
We consider first-passage percolation on the two-dimensional triangular lattice $\mathcal{T}$. Each site $v\in\mathcal{T}$ is assigned independently a passage time of either $0$ or $1$ with probability $1/2$. Denote by $B^+(0,n)$ the upper half-disk with radius $n$ centered at $0$, and by $c_n^+$ the first-passage time in $B^+(0,n)$ from $0$ to the half-circular boundary of $B^+(0,n)$. We prove
\[\lim_{n\rightarrow\infty}\frac{c_n^+}{\log n}=\frac{\sqrt{3}}{2\pi}~ a.s.,~\lim_{n\rightarrow\infty}\frac{E c_n^+}{\log n}=\frac{\sqrt{3}}{2\pi},~\lim_{n\rightarrow\infty}\frac{\mathrm{Var}(c_n^+)}{\log n}=\frac{2\sqrt{3}}{\pi}-\frac{9}{\pi^2}.\]
These results enable us to prove limit theorems with explicit constants for any first-passage time between boundary points of Jordan domains. In particular, we find the explicit limit theorems for the cylinder point to point and cylinder point to line first-passage times.
\end{abstract}
\maketitle
\section{Introduction}
Since Hammersley and Welsh \cite{HW65} introduced first-passage percolation (FPP) in 1965, this stochastic growth model has attracted much attention from mathematicians and physicists.  For the main results and recent developments of FPP, we refer the reader to the survey \cite{ADH16}, especially Section 3.7.1 there for critical FPP. For Bernoulli critical FPP on the triangular lattice,  the author in \cite{Yao16} derived the exact asymptotic behavior for the first-passage time from the center of a disk to its boundary. In this paper, we consider a boundary version of that result. Namely, we study the first-passage time in a half-disk from its center to  its half-circular boundary.

Here is an alternative description of our model. Construct a random maze on the hexagonal lattice by putting an obstacle on each hexagon according to the outcome of a fair coin toss. Consider a walker in the maze starting at the hexagon centered at 0. What is the minimum number of obstacles the walker has to cross to reach the circle of radius $n$ centered at 0 (interior point version)? Or, if the walker is only allowed to walk in the upper half-plane, what is the minimum number of obstacles the walker has to cross to reach the upper half-circle of radius $n$ centered at 0 (boundary point version)? From the results in \cite{Yao16} and this paper, one can see that when $n$ is large with high probability the latter quantity is approximately 3 times the former. See Figure \ref{fig1} for an illustration.

\begin{figure}
\begin{center}
\includegraphics[height=0.4\textwidth]{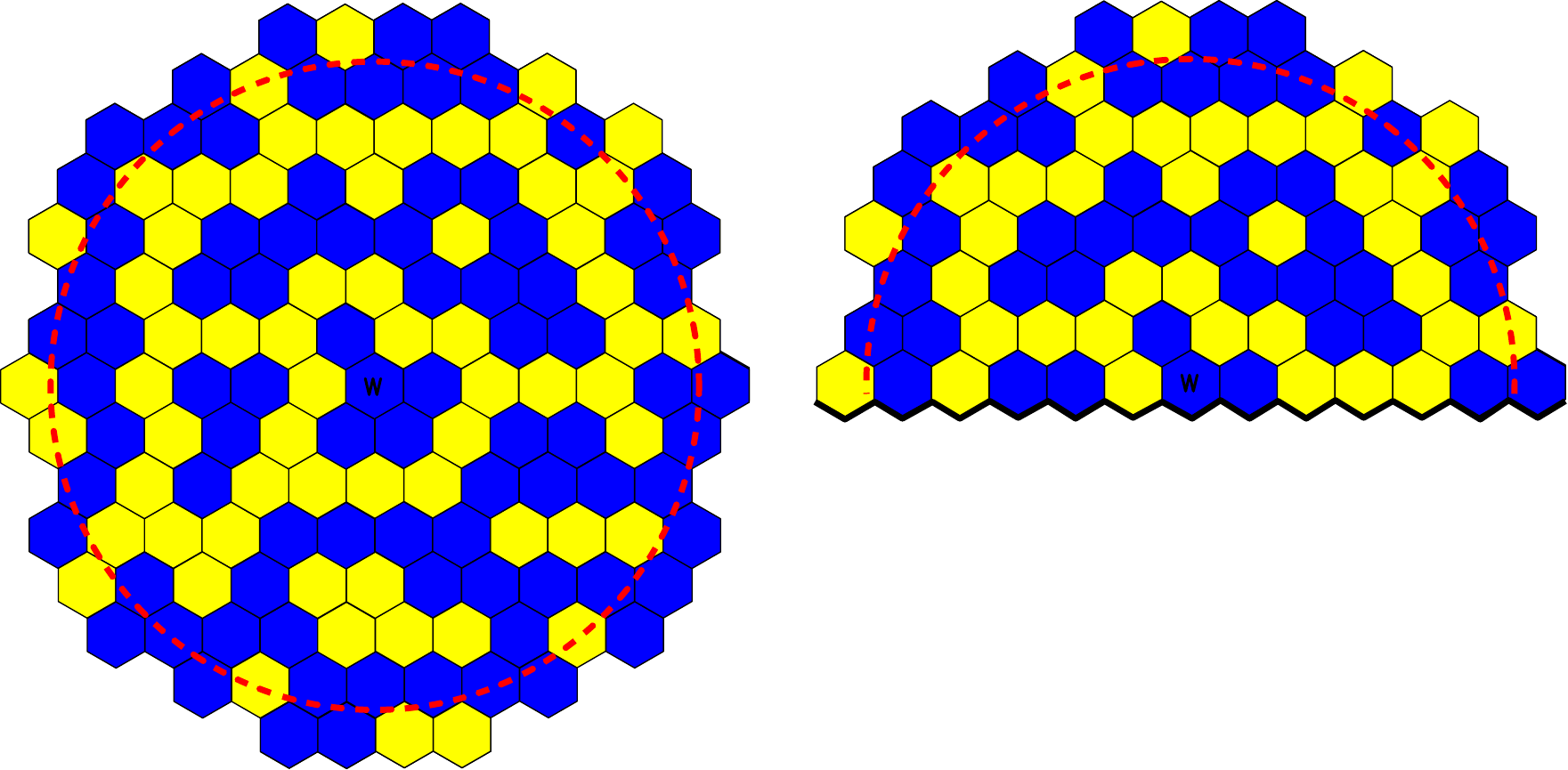}
\caption{All blue hexagons are void spaces and all yellow hexagons are
obstacles.  W means that the walker starts at the hexagon
centered at 0.  The walker has to cross at least 1 obstacle to
reach the red circle in the left maze, and cross at least 2
obstacles to reach the red half-circle in the right
maze.}\label{fig1}
\end{center}
\end{figure}

Let $\mathcal{T}=(\mathbf{V},\mathbf{E})$ be the two-dimensional triangular lattice, where $\mathbf{V}=\{x+ye^{i\pi/3}:x,y\in\mathbb{Z}\}$ is the set of sites, and $\mathbf{E}=\{\{x,y\}:x,y\in\mathbf{V},\|x-y\|_2=1\}$ is the set of bonds. Suppose $\{t(v):v\in\mathbf{V}\}$ is a family of i.i.d. Bernoulli random variables: each $t(v)$ takes the value $0$ or $1$ with equal probability. This is \textbf{critical site percolation} on $\mathcal{T}$, but we also call it \textbf{critical first-passage percolation} on $\mathcal{T}$ for reasons that will be clear below. We write $P$ for this critical percolation measure, and $E$ for the corresponding expectation. A \textbf{path} is a sequence $(v_0,v_1,\cdots,v_n)$ of distinct sites of $\mathcal{T}$ such that $\{v_{j-1},v_j\}\in\mathbf{E}$ for each $j=1,2,\cdots,n$. For a path $\gamma=(v_0,v_1,\cdots,v_n)$, define the \text{passage time} of $\gamma$ as
\[T(\gamma)=\sum_{j=0}^{n}t(v_j).\]

We also consider the dual of $\mathcal{T}$, the two-dimensional hexagonal lattice $\mathcal{H}=(\mathbf{V}_d,\mathbf{E}_d)$, such that each $v\in\mathbf{V}$ lies at the center of exactly one face (or hexagon) of $\mathcal{H}$. Each hexagon is assigned the same value as its center (which is a site in $\mathbf{V}$). For $n\in\mathbb{N}$, let $B^+(0,n)$ be the smallest connected domain of hexagons containing the closed upper half-disk centered at $0$ with radius $n$.  Let $\Delta_oB^+(0,n)$ be the half-circular boundary of $B^+(0,n)$, that is, the set of hexagons that do not belong to $B^+(0,n)$ but are adjacent to those hexagons (with centers lying on or above $x$-axis) intersecting the half-circle of radius $n$ centered at $0$.
Denote by $c_n^+$ the first-passage time in $B^+(0,n)$ between $0$ and $\Delta_oB^+(0,n)$. More precisely,
\begin{eqnarray*}
c_n^+:=\inf\{T(\gamma): \gamma\subseteq B^+(0,n) \text{ starting at }0 \text{ and ending at a neighbor } \text{of }\Delta_o B(0,n)\}.
\end{eqnarray*}

Our main theorem is
\begin{theorem}\label{thmmain}
\[\lim_{n\rightarrow\infty}\frac{c_n^+}{\log n}=\frac{\sqrt{3}}{2\pi}~ a.s.,~\lim_{n\rightarrow\infty}\frac{E c_n^+}{\log n}=\frac{\sqrt{3}}{2\pi},~\lim_{n\rightarrow\infty}\frac{\mathrm{Var}(c_n^+)}{\log n}=\frac{2\sqrt{3}}{\pi}-\frac{9}{\pi^2}.\]
\end{theorem}

\begin{remark}
For FPP corresponding to other critical percolation models (e.g., bond percolation on the square lattice \cite{Gri99} and Voronoi percolation \cite{BR06}), it is expected that Theorem \ref{thmmain} still holds, provided that the convergence of chordal exploration path to chordal SLE$_6$ is established. The reason is the following: similarly to the proof of Proposition \ref{propeqi}, one can show that $c_n^+$ equals the maximum number of disjoint yellow half-circuits surrounding $0$ in $B^+(0,n)$; then using the idea from Section 4 of \cite{BN11}, one can express the event $\{c_n^+\geq k\}$ in terms of the collection of all cluster interfaces.
\end{remark}

\begin{remark}
With an argument similar to (1.13) in \cite{KZ97} and Theorem 1.6 in
\cite{DLW16}, one can prove that
\begin{equation*}
\frac{c_n^+-Ec_n^+}{\sqrt{\mathrm{Var}(c_n^+)}}\rightarrow N(0,1)\mbox{ in
distribution as }n\rightarrow\infty,
\end{equation*}
where $N(0,1)$ is a standard normal random variable (with mean 0 and
variance 1).  This, together with Theorem \ref{thmmain}, implies that there exists
a function $\eta(n)$ with $\eta(n)\rightarrow 0$ as $n\rightarrow \infty$,
such that
\begin{equation*}
\frac{c_n^+-(1+\eta(n))\sqrt{3}\log
n/(2\pi)}{\sqrt{(2\sqrt{3}/\pi-9/(\pi)^2)\log n}}\rightarrow
N(0,1)\mbox{ in distribution as }n\rightarrow\infty.
\end{equation*}
We conjecture, but can not prove, that one may choose $\eta\equiv 0$.  Let us point out
that the explicit form of the CLT in Corollary 1.2 of \cite{Yao16} should
be replaced with a similar weaker form.
\end{remark}

Let $\mathcal{H}_{\delta}:=\delta\mathcal{H}$ be the two-dimensional hexagonal lattice with lattice spacing $\delta$. For $\alpha\in(0,2\pi)$, let $
\mathbb{D}^{\alpha}:=\{re^{i\theta}: 0<r<1,0<\theta<\alpha\}$ be the circular sector with center angle $\alpha$. Let $\mathbb{D}^{\alpha}_{\delta}$ be the smallest connected domain of hexagons (with lattice spacing $\delta$) containing $\overline{\mathbb{D}^{\alpha}}$. Consider critical percolation on $\mathbb{D}^{\alpha}_{\delta}$ and write $E_{\delta}$ for the corresponding expectation. Let $0_{\delta}$ be the center of a closest hexagon of $\mathbb{D}^{\alpha}_{\delta}$ to $0$. Define $c^+_{\delta}(\alpha)$ to be the first-passage time in $\mathbb{D}^{\alpha}_{\delta}$ between $0_{\delta}$ and the circular part of the boundary of $\mathbb{D}^{\alpha}_{\delta}$. Then Theorem \ref{thmmain} can be generalized to the following:
\begin{corollary}\label{cormain}
\[\lim_{\delta\downarrow0}\frac{c_{\delta}^+(\alpha)}{-\log \delta}=\frac{\sqrt{3}}{2\alpha} ~a.s.,~\lim_{\delta\downarrow0}\frac{E_{\delta} c_{\delta}^+(\alpha)}{-\log \delta}=\frac{\sqrt{3}}{2\alpha},~\lim_{\delta\downarrow0}\frac{\mathrm{Var}_{\delta}(c_{\delta}^+(\alpha))}{-\log \delta}=\frac{2\sqrt{3}}{\alpha}-\frac{9}{\pi\alpha}.\]
\end{corollary}

Suppose $D\subsetneq\mathbb{C}$ is a Jordan domain (i.e., $\partial D$ is a homeomorphism of the unit circle) and $a, b\in \partial D$. Let $\{\beta(t):t\geq 0\}$ be some continuous parametrization of $\partial D$. We assume $\partial D$ has both left and right tangent lines at $a$ and $b$, i.e., $\beta(t)$ has left and right derivatives at $t_0$ and $t_1$ where $\beta(t_0)=a$ and $\beta(t_1)=b$. Let $\Theta_D(a)$ be the angle subtended by the left and right tangent lines at $a$. Then $\Theta_D(a)\in(0,2\pi)$ and whether $\Theta_D(a)>\pi$ or not can be determined easily by comparing a small enough neighborhood of $a$ in $D$ with a circular sector. Define $\Theta_D(b)$ in the same way. Let $D_{\delta}$ be the smallest connected domain of hexagons containing $\overline{D}$. Denote by $T_{D_{\delta}}(a_{\delta}, b_{\delta})$ the first-passage time  in $D_{\delta}$ between $a_{\delta}$ and $b_{\delta}$. Then we have
\begin{proposition}\label{propbb}
\[\lim_{\delta\downarrow 0}\frac{T_{D_{\delta}}(a_{\delta}, b_{\delta})}{-\log\delta}=\frac{\sqrt{3}}{2\Theta_D(a)}+\frac{\sqrt{3}}{2\Theta_D(b)} \text{ in probability},~\lim_{\delta\downarrow 0}\frac{E_{\delta}T_{D_{\delta}}(a_{\delta}, b_{\delta})}{-\log\delta}=\frac{\sqrt{3}}{2\Theta_D(a)}+\frac{\sqrt{3}}{2\Theta_D(b)}.\]
\end{proposition}

\begin{remark}
We believe that a limit result for $\mathrm{Var}(T_{D_{\delta}}(a_{\delta}, b_{\delta}))$ is possible provided $\partial D$ is good (say, $\partial D$ is locally analytic at $a,b$). We leave the details to the interested reader.
\end{remark}

For $k\in\mathbb{Z}$, let
\[H_k=\{v\in\mathbf{V}: \mathrm{Im}(v)=\sqrt{3}k\}\]
be a hyperplane. As in \cite{SW78} and \cite{Kes86} (see also \cite{HW65}), for $m<n$ we define
\begin{align*}
&t_{m,n}=\inf\{T(\gamma):\gamma \text{ is a path from } (0,\sqrt{3}m) \text{ to } (0,\sqrt{3}n), \text{ except for its }\\
 &\hspace{0.72in} \text{ endpoints, lies strictly between }H_m \text{ and }H_n\},\\
&s_{m,n}=\inf\{T(\gamma):\gamma \text{ is a path from } (0,\sqrt{3}m) \text{ to some point in }H_n,\text{ except for its}\\
 &\hspace{0.72in} \text{ endpoints, lies strictly between }H_m \text{ and }H_n\}.
\end{align*}
$t_{0,n}$ and $s_{0,n}$ are called the \textbf{cylinder point to point} and \textbf{cylinder point to line first-passage times}, respectively.
Then we have
\begin{proposition}\label{propcyl}
\[\lim_{n\rightarrow\infty}\frac{s_{0,n}}{\log n}=\frac{\sqrt{3}}{2\pi}~a.s.,
~\lim_{n\rightarrow\infty}\frac{Es_{0,n}}{\log n}=\frac{\sqrt{3}}{2\pi},~
\lim_{n\rightarrow\infty}\frac{\mathrm{Var}(s_{0,n})}{\log n}=\frac{2\sqrt{3}}{\pi}-\frac{9}{\pi^2}.\]
\[\lim_{n\rightarrow\infty}\frac{t_{0,n}}{\log n}=\frac{\sqrt{3}}{\pi} \text{ in probability but not a.s.},\]
\[\lim_{n\rightarrow\infty}\frac{Et_{0,n}}{\log n}=\frac{\sqrt{3}}{\pi},
~\lim_{n\rightarrow\infty}\frac{\mathrm{Var}(t_{0,n})}{\log n}=\frac{4\sqrt{3}}{\pi}-\frac{18}{\pi^2}.\]
\end{proposition}

Our proof strategy for Theorem \ref{thmmain} is analogous to that of \cite{Yao16} where the first-passage time from the center of a disk to its boundary is studied.  Using a color switching technique we obtain that the first-passage time in a half-annulus has the same distribution as the number of (cluster) interface half-loops surrounding 0 in the half-annulus under a monochromatic boundary condition.  The well-known result that the percolation chordal exploration path converges weakly to chordal SLE$_6$ (see \cite{Smi01} and \cite{CN07}) tells us that those discrete interface half-loops converges weakly to the corresponding half-loops in the continuum. Then we use SLE techniques from \cite{Law15} to compute the distribution of the ``conformal radii" of the nested half-loops surrounding a fixed point. Surprisingly, this distribution is related to the distribution of conformal radii of CLE$_{24/5}$ whose moment generating function is derived in \cite{SSW09}.  This allows us to obtain explicit limit theorems for the half-loops in the continuum. Then the limit results for $c_n^+$ and $E[c_n^+]$ follow easily. Using Hongler and Smirnov's formula for the expected number of clusters in a rectangle \cite{HS11}, we give an alternative and more straightforward proof (using no SLE techniques) of the limit results for $c_n^+$ and $E[c_n^+]$.  Let us mention that it is possible to use the SLE techniques in Section 4.3 of \cite{Cur15} to give a third proof. In order to prove the limit result for Var$(c_n^+)$, we use a martingale method from \cite{KZ97}.

The organization of the paper is as follows. In Section \ref{secpre}, we give some definitions and relate the first-passage time in an annulus to the number of interface half-loops. In Section \ref{secconf}, we compute the exact distribution of ``conformal radii'' for the interface half-loops in the continuum. In Sections \ref{secfirst} and \ref{secsec}, we present two different proofs for the a.s. and $L^1$ convergences of $c_n^+/\log n$. In Section \ref{secvar}, we show the convergence of $\mathrm{Var}(c_n^+)/\log n$. In the last section, we complete the proofs of Theorem \ref{thmmain}, Corollary \ref{cormain}, Propositions \ref{propbb} and~\ref{propcyl}.

\section{Preliminaries}\label{secpre}
\subsection{Definitions and some discrete results}
Recall that $\mathcal{T}=(\mathbf{V},\mathbf{E})$ is the two-dimensional triangular lattice and  $\mathcal{H}=(\mathbf{V}_d,\mathbf{E}_d)$ is its dual.  The critical site percolation on $\mathcal{T}$ is an assignment of $0$ (equivalently, blue) or $1$ (equivalently, yellow) to each site of $\mathcal{T}$ (i.e., to each hexagon of $\mathcal{H}$). We denote the resulting probability space by $(\Omega,\mathscr{F},P)$ where $\Omega=\{0,1\}^{\mathbf{V}}$, and write $E$ for the corresponding expectation. Two hexagons are neighbors if they share a common edge. So a path $(v_0,v_1,\cdots,v_n)$ in $\mathcal{T}$ corresponds to a path $(h_0,h_1,\cdots,h_n)$ in $\mathcal{H}$ such that each $v_i$ lies at the center of the hexagon $h_i$ for each $0\leq i\leq n$. A path is called a \textbf{circuit} if its first and last sites (or hexagons) are neighbors.

We denote by $\mathbb{D}$ the unit disk in $\mathbb{C}$ centered at $0$. Define $\mathbb{H}$ to be the upper half-plane, i.e., $\mathbb{H}=\{z\in\mathbb{C}:\text{Im}(z)>0\}$. Let $\mathbb{D}^+:=\mathbb{D}\cap\mathbb{H}$ be the upper half unit disk. For $r>0$, let $\mathbb{D}_r^+:=r\mathbb{D}^+$ be the upper half-disk of radius $r$ centered at $0$. For $v\in\mathbf{V}$, denote by $B^+(v,r)$ the smallest connected domain of hexagons (in $\mathcal{H}$) which contains $\overline{v+\mathbb{D}_r^+}$. For $1\leq r<R$, let $A^+(r,R)$ be the discrete half-annulus centered at $0$ with inner radius $r$ and outer radius $R$. More precisely,
\[A^+(r,R):=B^+(0,R)\setminus B^+(0,r).\]
Let $\Delta A^+(r,R)$ be the \textbf{external site boundary} of $A^+(r,R)$, i.e., the set of hexagons that do not belong to $A^+(r,R)$ but are adjacent to hexagons in $A^+(r,R)$. $\Delta A^+(r,R)$ contains two paths of hexagons which lie under the $x$-axis (see Figure \ref{fig2}). We denote the left path by $\Delta_l A^+(r,R)$ and the right one by $\Delta_r A^+(r,R)$. The set of hexagons that are in $\Delta A^+(r,R)\setminus\{\Delta_l A^+(r,R)\cup \Delta_r A^+(r,R)\}$ and intersect $\mathbb{D}_r^+$ is denoted by $\Delta_i A^+(r,R)$, and $\Delta A^+(r,R)\setminus \{\Delta_l A^+(r,R)\cup \Delta_r A^+(r,R)\cup \Delta_i A^+(r,R)\}$ is denoted by $\Delta_o A^+(r,R)$. A path $(v_0,v_1,\cdots,v_n)$ in $A^+(r,R)$ is called a \textbf{half-circuit surrounding $0$} if $v_0$ has a neighbor in $\Delta_lA^+(r,R)$ and $v_n$ has a neighbor in $\Delta_rA^+(r,R)$. A \textbf{percolation cluster} is a maximal, connected and monochromatic subset of $\mathcal{T}$ (or $\mathcal{H}$). An \textbf{interface path} is a sequence $(e_0,e_1,\cdots,e_n)$ of distinct edges of $\mathcal{H}$ which belongs to the boundary of a cluster with $e_{i-1}$ and $e_i$ sharing a vertex of $\mathcal{H}$ for each $i=1,\cdots,n$. An \textbf{interface half-loop surrounding $0$} is defined in the obvious way. Define
\begin{align*}
&\rho^+(r,R):=\text{ the maximum number of disjoint yellow half-circuits surrounding } 0 \\
&\hspace{0.85in}\text{ in } A^+(r,R),\\
&N^+(r,R):=\text{ the number of interface half-loops surrounding }  0 \text{ in } A^+(r,R),\\
&T^+(r,R):=\inf\{T(\gamma): \gamma\in A^+(r,R), \text{ the first site of }\gamma \text{ has a neighbor in }\Delta_iA^+(r,R) \\
&\hspace{1.1in}\text{ and the last site of }\gamma \text{ has a neighbor in }\Delta_oA^+(r,R)\}.
\end{align*}
As Proposition 2.4 in \cite{Yao16}, we have
\begin{proposition}\label{propeqi}
Suppose $1\leq r<R$. Then we have
\begin{itemize}
\item
\[T^+(r,R)=\rho^+(r,R).\]
\item
Assume hexagons in $\Delta_oA^+(r,R)$ are blue. Then $T^+(r,R)$ has the same distribution as $N^+(r,R)$.
\end{itemize}
\end{proposition}
\begin{proof}
The proof is similar to that of Proposition 2.4 in \cite{Yao16}. Since the proof of the first item is standard (see, e.g., (2.39) in \cite{KZ97}), we describe the main idea for the proof of the second item. The key idea is to find a bijection between $\{\rho^+(r,R)=n\}$ and $\{N^+(r,R)=n\}$ for any $n\in\mathbb{N}\cup\{0\}$. The case when $n=0$ is trivial since the bijection is just the identity map. So we may just assume $n\geq 1$. Let $\omega\in\{\rho^+(r,R)=n\}$. We label all disjoint yellow half-circuits surrounding $0$ in $A^+(r,R)$ from outside to inside by $\mathcal{C}_1(\omega)$, $\mathcal{C}_2(\omega)$,$\cdots$, $\mathcal{C}_n(\omega)$. More precisely, $\mathcal{C}_1(\omega)$ is the outermost yellow half-circuit surrounding $0$ in $A^+(r,R)$, and $\mathcal{C}_2(\omega)$ is the outermost yellow half-circuit surrounding $0$ in the connected component of $A^+(r,R)\setminus\mathcal{C}_1(\omega)$ which is adjacent to $\Delta_iA^+(r,R)$, and so on. If $n$ is odd we flip the colors of hexagons in $\mathcal{C}_2(\omega)$, $\mathcal{C}_4(\omega)$, $\cdots$, $\mathcal{C}_{n-1}(\omega)$ and hexagons lying between $\Delta_iA^+(r,R)$ and $\mathcal{C}_n(\omega)$; if $n$ is even we flip the colors of hexagons in $\mathcal{C}_2(\omega)$, $\mathcal{C}_4(\omega)$, $\cdots$, $\mathcal{C}_{n}(\omega)$. It is easy to see that the resulting new configuration from the described color switching is in $\{N^+(r,R)=n\}$. But unfortunately, the map coming from such a color switching is not one-to-one and thus not a bijection. It turns out that one needs to flip the colors of more hexagons to construct a bijection. Namely, when $n$ is odd, besides switching the colors of hexagons in $\mathcal{C}_2(\omega)$, $\mathcal{C}_4(\omega)$, $\cdots$, $\mathcal{C}_{n-1}(\omega)$ and hexagons lying strictly between $\Delta_iA^+(r,R)$ and $\mathcal{C}_n(\omega)$, we also switch the colors of hexagons lying strictly between $\mathcal{C}_1(\omega)$ and $\mathcal{C}_2(\omega)$, $\mathcal{C}_3(\omega)$ and $\mathcal{C}_4(\omega)$,$\cdots$, $\mathcal{C}_{n-2}(\omega)$ and $\mathcal{C}_{n-1}(\omega)$; When $n$ is even, besides switching the colors of hexagons in $\mathcal{C}_2(\omega)$, $\mathcal{C}_4(\omega)$, $\cdots$, $\mathcal{C}_{n}(\omega)$, we also switch the colors of hexagons lying between $\mathcal{C}_1(\omega)$ and $\mathcal{C}_2(\omega)$, $\mathcal{C}_3(\omega)$ and $\mathcal{C}_4(\omega)$,$\cdots$, $\mathcal{C}_{n-1}(\omega)$ and $\mathcal{C}_{n}(\omega)$. One can check the map coming from this new color switching is a bijection. See Figure \ref{fig2} for an example. We refer the reader to \cite{Yao16} for more details.
\end{proof}

\begin{figure}
\begin{center}
\includegraphics[height=0.23\textwidth]{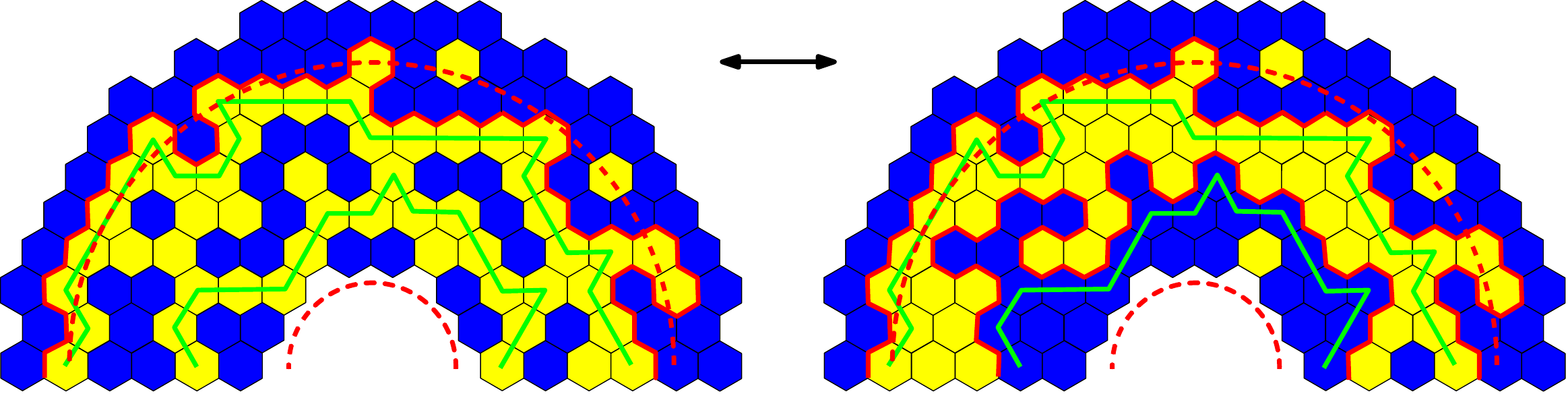}
\caption{Color switching. The left annulus has two disjoint yellow half-circuits surrounding $0$ while the right one has two disjoint interface half-loops surrounding $0$.}\label{fig2}
\end{center}
\end{figure}

The following lemma is a large deviation bound for $T^{+}(r,R)$.
\begin{lemma}\label{lemld}
There exist constants $C_1,C_2>0$ and $K>1$ such that for all $1\leq r<R$ and $x>K\log_2(R/r)$,
\[P(T^+(r,R)\geq x)\leq C_1e^{-C_2x}.\]
\end{lemma}
\begin{proof}
The proof is similar to that of Corollary 2.3 in \cite{Yao14}.
\end{proof}

\subsection{Chordal SLE}
We give a brief introduction of chordal Schramm-Loewner evolution (SLE). Please refer to \cite{Law05} for more about SLE.
Let $(B_t)_{t\geq0}$ be a standard Brownian motion on $\mathbb{R}$ with $B_0=0$. Let $\kappa\geq 0$ and consider the solution to the chordal Loewner equation for the upper half plane,
\begin{equation}\label{eqloe}
\partial_tg_t(z)=\frac{2}{g_t(z)-\sqrt{\kappa}B_t},~~~g_0(z)=z, z\in\overline{\mathbb{H}}.
\end{equation}
This is well defined as long as $g_t(z)-\sqrt{\kappa}B_t\neq 0$, i.e., for all $t<T_z$, where $T_z:=\inf\{t\geq 0:g_t(z)-\sqrt{\kappa}B_t=0\}$. For each $t>0$, $g_t:\mathbb{H}\setminus K_t\rightarrow\mathbb{H}$ is a conformal map, where $K_t:=\{z\in\overline{\mathbb{H}}:T_z\leq t\}$ is a compact subset of $\overline{\mathbb{H}}$ such that $\mathbb{H}\setminus K_t$ is simply connected. It is known (see \cite{RS05}) that $\gamma(t):=g_t^{-1}(\sqrt{\kappa}B_t)$ exists and is continuous in $t$, and the curve $\gamma$ is called the \textbf{trace} of chordal SLE$_{\kappa}$. It is also proven in the same paper that $\gamma$ is simple if and only if $\kappa\in[0,4]$.

Let $D$ be a simply connected domain and $a,b$ be distinct points on $\partial D$. Let $f:\mathbb{H}\rightarrow D$ be a conformal map with $f(0)=a$ and $f(\infty)=b$. If $\gamma$ is the chordal SLE$_{\kappa}$ trace in $\overline{\mathbb{H}}$ from $0$ to $\infty$, then $f\circ\gamma$ defines the chordal SLE$_{\kappa}$ trace from $a$ to $b$ in $\overline{D}$.

\section{Conformal radii}\label{secconf}
In \cite{CN06}, the scaling limit of the interface loops of critical site percolation on $\mathcal{T}$ was constructed and proved. This scaling limit is the conformal loop ensemble (CLE) defined in \cite{She09} for $\kappa=6$. In \cite{SSW09}, the distribution of the conformal radii of the nested loops from CLE$_{\kappa}$ ($8/3\leq \kappa\leq 8$) in $\mathbb{D}$ surrounding $0$ was calculated. In this section, we are interested in the nested interface half-loops (in the scaling limit) surrounding $0$ in $\mathbb{D}^+$. We will compute the distribution of the ``conformal radii'' of those nested half-loops. Our derivation applies to any SLE$_{\kappa}$ where $\kappa>4$. But we can only obtain the explicit moment generating function for $\kappa=6$, which is what we need for this paper. For $\kappa>4$ and $\kappa\neq 6$, we show that it has the same distribution as the first hitting time of some SDE (see Remark \ref{remsde} below), which might be interesting in itself.

Let $D_0:=\mathbb{D}^+$ be the upper half unit disk. Let $l_0:=\inf\{x:x\in\partial D_0\cap\mathbb{R}\}=-1$ and $r_0:=\sup\{x:x\in\partial D_0\cap\mathbb{R}\}=1$. Suppose $\{\gamma(t),t\geq0\}$ is the chordal SLE$_6$ trace in $\overline{D}_0$ from $-1$ to $0$ and $\tau_0$ is the first time $t$ that $0$ and $r_0$ are in distinct components of $\overline{D}_0\setminus\gamma[0,t]$. Let $D_1$ be the connected component of $D_0\setminus\gamma[0,\tau_0]$ that contains $0$ as a boundary point. We inductively define $(D_{k+1},l_k,r_k)$ in the following way. For $k\in \mathbb{N}$, let $l_k:=\inf\{x:x\in\partial D_k\cap\mathbb{R}\}$ and $r_k:=\sup\{x:x\in\partial D_k\cap\mathbb{R}\}$. If $k$ is odd, then denote by $\tau_k$ the first time $t$ that $0$ and $l_k$ are in distinct components of $\overline{D}_k\setminus\gamma[\tau_{k-1},t]$; If $k$ is even, then denote by $\tau_k$ the first time $t$ that $0$ and $r_k$ are in distinct components of $\overline{D}_k\setminus\gamma[\tau_{k-1},t]$. Let $D_{k+1}$ be the connected component of $D_k\setminus\gamma[\tau_{k-1},\tau_k]$ that contains $0$ as a boundary point. If $D$ is a simply connected domain and $z\in D$, define $\text{CR}(D,z)$ to be the \textbf{conformal radius} of $D$ viewed from $z$, i.e., $\text{CR}(D,z)=|g^{\prime}(z)|^{-1}$ where $g$ is any conformal map from $D$ to the unit disk $\mathbb{D}$ that sends $z$ to $0$. For $k\in\mathbb{N}\cup\{0\}$, let $\tilde{D}_k$ be the reflected domain of $D_k$, that is,
\[\tilde{D}_k=\{z\in\mathbb{C}:z\in D_k \text{ or }\bar{z}\in D_k\}\cup(l_k,r_k).\]
Define
\[Z_k:=\log\text{CR}(\tilde{D}_{k-1},0)-\log\text{CR}(\tilde{D}_k,0),~k\in\mathbb{N}.\]

We denote by $\mathbb{P}$ the probability measure associated with the chordal SLE$_6$ in $D_0$, and $\mathbb{E}$ for the corresponding expectation.
Then we have
\begin{theorem}\label{thm2}
Suppose $\kappa=6$. The $Z_k$'s are i.i.d. random variables and
\[\mathbb{E}e^{\lambda Z_k}=\frac{\sqrt{3}}{2\cos(\pi \sqrt{1/36+2\lambda/3})},\]
where $\mathrm{Re}(\lambda)<1/3$.
\end{theorem}
Theorem \ref{thm2} implies the following corollary immediately.
\begin{corollary}
Suppose $\kappa=6$. For $k\in\mathbb{N}$,
\[\mathbb{E}Z_k=\frac{2\sqrt{3}\pi}{3},~\mathrm{Var}(Z_k)=\frac{16\pi^2}{3}-8\sqrt{3}\pi.\]
\end{corollary}
The proof of Theorem \ref{thm2} relies on ideas from \cite{Law15} and \cite{SSW09}. We first introduce the following easy but useful lemma.
\begin{lemma}\label{lemcon}
If $\phi: D\rightarrow D^{\prime}$ is a conformal bijection and $z\in D$, then
\[\mathrm{CR}(D^{\prime}, \phi(z))=|\phi^{\prime}(z)|\mathrm{CR}(D,z).\]
\end{lemma}
\begin{proof}
The proof follows from the definition of conformal radius.
\end{proof}
\begin{proof}[Proof of Theorem \ref{thm2}]
For $k$ odd, by the domain Markov property and the locality property for SLE$_6$ (see, e.g., Proposition 6.14 in \cite{Law05}) , $\gamma[\tau_{k-1},\tau_k]$ has the same distribution as a chordal SLE$_6$ trace in $D_k$ from $r_k$ to $l_k$ stopped when it disconnects $0$ from $l_k$. Similarly, for $k$ even, $\gamma[\tau_{k-1},\tau_k]$ has the same distribution as a chordal SLE$_6$ trace in $D_k$ from $l_k$ to $r_k$ stopped when it disconnects $0$ from $r_k$.
That the $Z_k$'s are i.i.d. follows from Lemma \ref{lemcon} and the conformal invariance of chordal SLE$_6$ (by definition). So it suffices to prove that $Z_1$ has the right moment generating function. Note that $f(z)=(\frac{z+1}{1-z})^2$ is a conformal bijection from $D_0$ to $\mathbb{H}$ with $f(-1)=0,f(0)=1,f(1)=\infty$. So the proof of Theorem \ref{thm2} is completed if we can show the following proposition, since $\mathrm{CR}(\tilde{D}_1,0)=\mathrm{CR}(\tilde{U}_{T_1},1)/4$ where the latter is defined below.
\end{proof}
\begin{proposition}
Suppose $\{\gamma(t),t\geq 0\}$ is a chordal SLE$_6$ trace in $\overline{\mathbb{H}}$ from $0$ to $\infty$. For $z\in\overline{\mathbb{H}}$, let $T_z:=\inf\{t\geq 0:g_t(z)-\sqrt{6}B_t=0\}$. Let $F_t:=(-\infty,0]\cup\{z\in\overline{\mathbb{H}}:T_z\leq t\}$ and $\tilde{F}_t=\{z:z\in F_t\text{ or }\bar{z}\in F_t\}$. For $t<T_1$, define $\tilde{U}_t$ to be the connected component of $\mathbb{C}\setminus \tilde{F}_t$ that contains $1$. Define $\mathrm{CR}(\tilde{U}_{T_1},1):=\lim_{t\uparrow T_1}\mathrm{CR}(\tilde{U}_{t},1)$ (see the Remark \ref{rem2} below). Then we have
\[\mathbb{E}e^{-\lambda \log(\mathrm{CR}(\tilde{U}_{T_1},1)/4)}=\frac{\sqrt{3}}{2\cos(\pi \sqrt{1/36+2\lambda/3})},\]
where $\mathrm{Re}(\lambda)<1/3$.
\end{proposition}
\begin{remark}\label{rem2}
For $t<T_1$, define the inradius of $\tilde{U}_t$ with respect to $1$ as $\mathrm{inrad}(\tilde{U}_t,1):=\inf\{|z-1|:z\notin\tilde{U}_t\}$. Then the Schwarz Lemma and the Koebe $1/4$ Theorem give
\[\mathrm{inrad}(\tilde{U}_t,1)\leq\mathrm{CR}(\tilde{U}_t,1)\leq 4\mathrm{inrad}(\tilde{U}_t,1).\]
The equation \eqref{eq21} below implies $\mathrm{CR}(\tilde{U}_t,1)$ is strictly increasing for $t<T_1$, so the limit $\lim_{t\uparrow T_1}\mathrm{CR}(\tilde{U}_{t},1)$ is well-defined and finite a.s.
\end{remark}
\begin{remark}\label{remsde}
For a general chordal SLE$_{\kappa}$ where $\kappa>4$, all quantities defined in the proposition and its proof are still well-defined. In this general case, \eqref{eq24} below becomes
\[d\theta_t=\frac{4}{5\sin(\theta_t/2)}\left[3-\frac{\kappa}{2}+\left(\frac{\kappa}{4}-1\right)\cos\left(\theta_t/2\right)\right]dt-\sqrt{\frac{4\kappa}{5}}dW_t, \text{ if } 0<\theta_t<2\pi.\]
Computing the exact distribution of the first hitting time of $2\pi$ for this process started at $0$ might be hard.
\end{remark}
\begin{proof}
As in the proof of Theorem 1 of \cite{Law15}, we define $x_t$ to be the rightmost point of $\tilde{F}_t\cap\mathbb{R}$ and $g_t(x_t)$ is defined by
\[g_t(x_t):=\inf\{g_t(x):x>0,T_x>t\}.\]
Let $X_t=g_t(1)-\sqrt{6}B_t$, $O_t=g_t(x_t)-\sqrt{6}B_t$, $Y_t=X_t-O_t$, $J_t=Y_t/X_t$, $\Upsilon_t=\mathrm{CR}(\tilde{U}_t,1)/4$. For $t<T_1$, it is not hard to find a conformal bijection between $\tilde{U}_t$ and $\mathbb{D}$ . Using such a conformal bijection one gets
\[\Upsilon_t=\frac{Y_t}{g_t^{\prime}(1)}.\]
By the Loewner equation \eqref{eqloe}, we have for $t<T_1$
\begin{eqnarray}
&&\partial_t\Upsilon_t=\frac{-2\Upsilon_tJ_t}{X_t^2(1-J_t)},\label{eq21}\\
&&dJ_t=\frac{J_t}{X_t^2}\left(4-\frac{2}{1-J_t}\right)dt+\frac{\sqrt{6}J_t}{X_t}dB_t.\label{eq22}
\end{eqnarray}
Define the random time change
\begin{equation}\label{eq23}
\sigma(t)=\inf\{s: s\geq 0,\Upsilon_s=e^{-2t/5}\}.
\end{equation}
We also define $\hat{\Upsilon}_t=\Upsilon_{\sigma(t)}=e^{-2t/5}$, $\hat{X}_t=X_{\sigma(t)}$, $\hat{J}_t=J_{\sigma(t)}$. The equation \eqref{eq21} and the chain rule imply
\[-\frac{2}{5}\hat{\Upsilon}_t=-\frac{2}{5}e^{-2t/5}=\partial_t\hat{\Upsilon}_t=\dot{\sigma}(t)\frac{-2\hat{\Upsilon}_t\hat{J}_t}{\hat{X}_t^2(1-\hat{J}_t)}.\]
Therefore,
\[\dot{\sigma}(t)=\frac{\hat{X}_t^2(1-\hat{J}_t)}{5\hat{J}_t}.\]
We change time in \eqref{eq22} to get
\[d\hat{J}_t=\frac{2-4\hat{J}_t}{5}dt+\sqrt{\frac{6\hat{J}_t(1-\hat{J}_t)}{5}}dW_t,\]
where $W_t=\int_0^{\sigma(t)}\frac{1}{\sqrt{\dot{\sigma}(\sigma^{-1}(s))}}dB_s$ is a standard Brownian motion. If we make the change of variables $\hat{J}_t=\frac{1+\cos(\theta_t/2)}{2}$, then It\^{o}'s formula implies
\begin{equation}\label{eq24}
d\theta_t=\frac{2}{5}\cot(\theta_t/2)dt-\sqrt{\frac{24}{5}}dW_t,\  \theta_0=0,\ \text{if } 0<\theta_t<2\pi.
\end{equation}
Note that $\theta_t$ behaves like a Bessel process, and it is reflected instantaneously at $0$ (reflected in the same way that the Bessel process is reflected). It is easy to see that $J_0=\hat{J}_0=1$ and $\lim_{t\uparrow T_1}J_t=0$. Let $\tau_0:=\inf\{s:\hat{J}_s=0\}$ and $S_{2\pi}:=\inf\{s:\theta_s=2\pi\}$. Then the definition of $\theta_t$ and \eqref{eq23} give
\begin{equation}\label{eq25}
\tau_0\overset{d}{=}S_{2\pi},\ \Upsilon_{T_1}\overset{d}{=}e^{-2\tau_0/5}.
\end{equation}
Our $\theta_t$ defined in \eqref{eq24} has the same distribution as $\theta_t$ for $\kappa=24/5$ defined in equation (6) of \cite{SSW09}. Hence Proposition 2 and the equation (3) from \cite{SSW09} say
\[\mathbb{E}e^{\lambda S_{2\pi}}=\frac{\sqrt{3}}{2\cos(\pi\sqrt{1/36+5\lambda/3})},~\mathrm{Re}(\lambda)<\frac{2}{15}.\]
The above displayed equation and \eqref{eq25} imply
\[\mathbb{E}e^{-\lambda \log(\mathrm{CR}(\tilde{U}_{T_1},1)/4)}=\mathbb{E}e^{-\lambda \log(\Upsilon_{T_1})}=\mathbb{E}e^{(2\lambda/5) S_{2\pi}}=\frac{\sqrt{3}}{2\cos(\pi \sqrt{1/36+2\lambda/3})},~\mathrm{Re}(\lambda)<\frac{1}{3},\]
which completes the proof of the proposition.
\end{proof}

\section{First proof of SLLN using conformal radii}\label{secfirst}
Our first proof of the strong law of large numbers (SLLN) for $c_n^+$ uses the conformal radii result that we proved in the last section. Recall the definition of $\tilde{D}_k$ in the previous section. For $\epsilon\in (0,1)$, we define
\[N(\epsilon):=\sup\{k:\overline{\mathbb{D}}_{\epsilon}\subseteq \tilde{D}_k\}.\]

Theorem \ref{thm2} and its corollary enable us to show the following
\begin{proposition}\label{propn}
\[\lim_{\epsilon\downarrow 0}\frac{N(\epsilon)}{-\log(\epsilon)}=\frac{\sqrt{3}}{2\pi} a.s.,~\lim_{\epsilon\downarrow 0}\frac{\mathbb{E}N(\epsilon)}{-\log(\epsilon)}=\frac{\sqrt{3}}{2\pi},~\lim_{\epsilon\downarrow0}\frac{\mathrm{Var}(N(\epsilon))}{-\log(\epsilon)}=\frac{2\sqrt{3}}{\pi}-\frac{9}{\pi^2}.\]
\end{proposition}
\begin{proof}
The proof uses some basic properties for renewal processes and a similar proof can be found in Proposition 3.2 of \cite{Yao16}.
\end{proof}
Next, we prove the scaling limit of $T^+(\tau r,\tau R)$ as $\tau\rightarrow\infty$.
\begin{proposition}\label{propnn}
Suppose $1\leq r<R$, $\tau\geq 1$ and $k\in\mathbb{N}$. Assume that hexagons in $\Delta_oA^+(\tau r,\tau R)$ are colored blue. We have
\[T^+(\tau r,\tau R)\rightarrow N(r/R) \text{ in distribution as } \tau\rightarrow\infty,\]
\[E(T^+\left(\tau r,\tau R\right)^k)\rightarrow \mathbb{E}(N(r/R)^k) \text{ as }\tau\rightarrow\infty.\]
\end{proposition}
\begin{proof}
Recall $\mathbb{D}^+$ is the upper half unit disk. Denote by $\mathbb{D}^+_{\delta}$ the smallest connected domain of hexagons  (in $\mathcal{H}_{\delta}$) containing $\overline{\mathbb{D}^+}$. Let $\partial \mathbb{D}^+_{\delta}$ be the topological boundary of $\mathbb{D}^+_{\delta}$ (here $\mathbb{D}^+_{\delta}$ is considered as a domain of $\mathbb{C}$)  and $\Delta \mathbb{D}^+_{\delta}$ be the external site boundary of $\mathbb{D}^+_{\delta}$ (i.e., the set of hexagons that do
not belong to $\mathbb{D}^+_{\delta}$ but are adjacent to hexagons in $\mathbb{D}^+_{\delta}$).  A vertex $x\in \partial \mathbb{D}^+_{\delta}$ is called an \textbf{e-vertex} if the edge containing $x$ that is not in  $\partial \mathbb{D}^+_{\delta}$ does not belong to $\mathbb{D}^+_{\delta}$ either. Let $(-1)_{\delta}$ ($0_{\delta}$, respectively) be a closest e-vertex of $\mathcal{H}_{\delta}$ to $-1$ ($0$, respectively). Denote by $\partial_{-1,0}\mathbb{D}^+_{\delta}$  the portion of $\partial \mathbb{D}^+_{\delta}$ traversed counterclockwise from $(-1)_{\delta}$ to $0_{\delta}$, and the portion of $\Delta \mathbb{D}^+_{\delta}$ whose hexagons are adjacent to  $\partial_{-1,0}\mathbb{D}^+_{\delta}$ is denoted by $\Delta_{-1,0}\mathbb{D}^+_{\delta}$. The remaining part of $\Delta \mathbb{D}^+_{\delta}$ is denoted by $\Delta_{0,-1}\mathbb{D}^+_{\delta}$.   Suppose we color yellow all hexagons in $\Delta_{-1,0}\mathbb{D}^+_{\delta}$ and blue all those in $\Delta_{0,-1}\mathbb{D}^+_{\delta}$. Then for any percolation configuration inside $\mathbb{D}^+_{\delta}$, there is a unique interface path (say $\gamma_{\delta}$) from $(-1)_{\delta}$ to $0_{\delta}$, which separates the yellow cluster adjacent to $\Delta_{-1,0}\mathbb{D}^+_{\delta}$ from the blue cluster adjacent to $\Delta_{0,-1}\mathbb{D}^+_{\delta}$. The random path $\gamma_{\delta}$ is called a \textbf{chordal exploration path} in $D_{\delta}$ from $(-1)_{\delta}$ to $0_{\delta}$. We remark that $\gamma_{\delta}$ does not depend on the color of hexagons in $\Delta\mathbb{D}^+_{\delta}$. It is well-known that $\gamma_{\delta}$ converges weakly to a chordal SLE$_6$ trace in $\mathbb{D}^+_{\delta}$ from $-1$ to $0$ (see \cite{Smi01} and \cite{CN07}).

By using 3-arm event in the half-plane, it is not hard to show that whenever $\gamma_{\delta}$ comes close to the boundary $\partial\mathbb{D}^+_{\delta}$ then it does touch the boundary with high probability (see, e.g., Lemma 6.1 of \cite{CN06}). This implies that the number of interface half-loops surrounding $(r/R)\mathbb{D}^+_{\delta}$  in  $\mathbb{D}^+_{\delta}$ converges weakly to $N(r/R)$ as $\delta\downarrow0$. Therefore $N^+(\tau r,\tau R)$ converges weakly to $N(r/R)$ as $\tau\rightarrow\infty$.

Now Proposition \ref{propeqi} implies $T^+(\tau r,\tau R)$ converges weakly to $N(r/R)$ as $\tau\rightarrow\infty$. Lemma~ \ref{lemld} says that $\{T^+(\tau r,\tau R)\}_{\tau\geq 1}$ is uniformly integrable, so
\[E(T^+(\tau r,\tau R)^k)\rightarrow \mathbb{E}(N(r/R)^k) \text{ as }\tau\rightarrow\infty.\]
\end{proof}

We are ready to prove the strong law of large numbers for $c_n^+$.
\begin{proposition}\label{propslln}
\[\lim_{n\rightarrow\infty}\frac{c_n^+}{\log n}=\frac{\sqrt{3}}{2\pi} a.s.,~\lim_{n\rightarrow\infty}\frac{E c_n^+}{\log n}=\frac{\sqrt{3}}{2\pi}.\]
\end{proposition}
\begin{proof}
The proof uses Propositions \ref{propn} and \ref{propnn}, and is similar to that of Proposition 3.6 of \cite{Yao16}.
\end{proof}

\section{Second proof of SLLN  using expected number of clusters}\label{secsec}
In \cite{HS11}, an explicit formula for the scaling limit of the expected number of clusters crossing a Jordan domain is given, and this scaling limit is proved to be conformal invariant. Using that result and Proposition \ref{propeqi}, we give a second proof of Proposition \ref{propslln}.

Let $D\subsetneq\mathbb{C}$ be a Jordan domain (i.e., $D$ is simply connected and the boundary of $D$, $\partial D$, is a Jordan curve). Let $z_1, z_2,z_3,z_4\in \partial D$. We assume $\partial D$ is oriented counterclockwise, and $z_1, z_2,z_3,z_4$ appear in this order. Suppose $\phi: D\rightarrow\mathbb{H}$ is a conformal map. The \textbf{cross-ratio} of $(D;z_1,z_2,z_3,z_4)$ is defined by
\begin{equation*}
\lambda:=\lambda(D;z_1,z_2,z_3,z_4)=\frac{(\phi(z_4)-\phi(z_3))(\phi(z_2)-\phi(z_1))}{(\phi(z_4)-\phi(z_2))(\phi(z_3)-\phi(z_1))}.
\end{equation*}
It is easy to see that M\"{o}bius transformations preserve cross-ratios, $\lambda\in(0,1)$ and $$\lambda(D;z_1,z_2,z_3,z_4)=1-\lambda(D;z_2,z_3,z_4,z_1).$$

We state a result by Hongler and Smirnov \cite{HS11}. Recall $D_{\delta}$ is the smallest connected domain of hexagons (in $\mathcal{H}_{\delta}$) which contains $\overline{D}$. Consider critical site percolation on $D_{\delta}$. Recall that $P_{\delta}$ and $E_{\delta}$ are the corresponding probability measure and expectation. Let $z_i^{\delta}$ be a closest vertex of $\partial D_{\delta}$ to $z_i$ for $i=1,2,3,4$. Let $N(D_{\delta};z_1^{\delta},z_2^{\delta},z_3^{\delta},z_4^{\delta})$ be the number of open clusters in $D_{\delta}$ which connect the arc along $\partial D_{\delta}$ from $z_1^{\delta}$ to $z_2^{\delta}$ and the arc along $\partial D_{\delta}$ from $z_3^{\delta}$ to $z_4^{\delta}$. Then we have
\begin{theorem}[Proposition 1 in \cite{HS11}]\label{thm1}
Let $D\subsetneq\mathbb{C}$ be a Jordan domain and $z_1, z_2,z_3,z_4\in \partial D$ are ordered counterclockwise. Suppose $\lambda$ is the cross-ratio of $(D;z_1,z_2,z_3,z_4)$. Then we have
\begin{eqnarray*}
\lim_{\delta\rightarrow 0}E_{\delta}(N(D_{\delta};z_1^{\delta},z_2^{\delta},z_3^{\delta},z_4^{\delta}))&=&\frac{2\pi\sqrt{3}}{\Gamma(1/3)}\lambda^{\frac{1}{3}}~_2F_1
\left(\frac{1}{3},\frac{2}{3};\frac{4}{3};\lambda\right)-\frac{\sqrt{3}}{4\pi}\lambda~_3F_2\left(1,1,\frac{4}{3};\frac{5}{3},2;\lambda\right)\\
&&+\frac{\sqrt{3}}{4\pi}\log\left(\frac{1}{1-\lambda}\right).
\end{eqnarray*}
\end{theorem}
\begin{remark}\label{rem1}
By the argument leading to the above theorem in \cite{HS11}, one sees that
\begin{equation*}
\left|\frac{2\pi\sqrt{3}}{\Gamma(1/3)}\lambda^{\frac{1}{3}}~_2F_1
\left(\frac{1}{3},\frac{2}{3};\frac{4}{3};\lambda\right)-\frac{\sqrt{3}}{4\pi}\lambda~_3F_2\left(1,1,\frac{4}{3};\frac{5}{3},2;\lambda\right)\right|\leq1.
\end{equation*}
\end{remark}

Let $a^+(1,R):=\{z\in\mathbb{H}:1<|z|<R\}$ for some $R>1$ be the half-annulus in the upper half-plane with inner radius $1$ and outer radius $R$. Let $z_1=-R, z_2=-1, z_3=1, z_4=R\in\partial a^+(1,R)$. Then we have
\begin{proposition}\label{prop6}
\begin{equation*}
\lim_{R\rightarrow\infty}\frac{\lim_{\delta\rightarrow0}E_{\delta}(N(a^+(1,R)_{\delta};z_1^{\delta},z_2^{\delta},z_3^{\delta},z_4^{\delta}))}{\log R}=\frac{\sqrt{3}}{4\pi}
\end{equation*}
\end{proposition}
\begin{proof}
Let $D((\log R)/\pi):=\{z\in\mathbb{C}:0<\text{Re}(z)<(\log R)/\pi, 0<\text{Im}(z)<1\}$. Then $\phi:a^+(1,R)\rightarrow D((\log R)/\pi)$ such that $\phi(z)=(\log z)/\pi$ is a conformal map. Hence
\begin{eqnarray*}
\lambda(a^+(1,R);z_1,z_2,z_3,z_4)&=&\lambda(D((\log R)/\pi);(\log R)/\pi+i,i,0,(\log R)/\pi)\\
&=&1-\lambda(D((\log R)/\pi);i,0,(\log R)/\pi,(\log R)/\pi+i).
\end{eqnarray*}
The proposition follows from Theorem \ref{thm1}, Remark \ref{rem1} and the following lemma.
\end{proof}
\begin{lemma}
Let $\eta$ be the \textbf{aspect-ratio} of the rectangle $D(\eta):=\{z\in\mathbb{C}:0<\text{Re}(z)<\eta,0<\text{Im}(z)<1\}$, i.e., the ratio of the width of the rectangle to its height. Then
\begin{equation*}
\lim_{\eta\rightarrow\infty}\frac{\log\lambda(D(\eta))}{\eta}=-\pi,
\end{equation*}
where $\lambda(D(\eta))$ is the cross-ratio of $(D(\eta);i,0,\eta,\eta+i)$.
\end{lemma}
\begin{proof}
To be consistent with the tradition in the literature, we will denote by $k$ some real number in $(0,1)$ in the proof. Let $R(k)$ be the rectangle with corners $\pm K(k^2)$, $\pm K(k^2)+iK(1-k^2)$ where
\begin{equation*}
K(u)=\int_0^1\frac{dt}{\sqrt{(1-t^2)(1-ut^2)}}
\end{equation*}
is the complete elliptic integral of the first kind.
Let $\psi:\mathbb{H}\rightarrow R(k)$ be the Schwartz-Christoffel transfrom
\begin{equation*}
\psi(z)=\int_0^z\frac{dt}{\sqrt{(1-t^2)(1-k^2t^2)}}.
\end{equation*}
Then $\psi$ is a conformal map with $\psi(-k^{-1})=-K(k^2)+iK(1-k^2)$, $\psi(-1)=-K(k^2)$, $\psi(1)=K(k^2)$ and $\psi(k^{-1})=K(k^2)+iK(1-k^2)$. So
$$\lambda(R(k);-K(k^2)+iK(1-k^2),-K(k^2),K(k^2),K(k^2)+iK(1-k^2))=\frac{(1-k)^2}{(1+k)^2}.$$
Note that the aspect-ratio for $R(k)$ is $2K(k^2)/K(1-k^2)$. To prove the lemma, it suffices to show
\begin{equation*}
\lim_{k\uparrow 1}\frac{\log((1-k)^2/(1+k)^2)}{2K(k^2)/K(1-k^2)}=-\pi.
\end{equation*}
The last limit holds because of the well-known estimates (see, e.g. p. 57 of \cite{MM99})
\begin{equation*}
K(1-k^2)=\frac{\pi}{2}+o(1),\  K(k^2)=2\log 2-\log\sqrt{1-k^2}+o(1),\text{ as } k\uparrow 1.
\end{equation*}
\end{proof}

Now we can give a second proof of the strong law of large numbers for $c_n^+$ (i.e., Proposition \ref{propslln}). For any $k, j \in\mathbb{N}$, let us put monochromatic (blue) boundary condition on $A^+(2^{k(j-1)},2^{kj})$ (i.e., color blue all hexagons in $\Delta A^+(2^{k(j-1)},2^{kj})$). In the rest of the section, we will fix the lattice spacing to be 1 (i.e., $\delta=1$) and enlarge the domain to get a scaling limit. It is easy to see from Proposition \ref{propeqi} that
\[T^+(2^{k(j-1)},2^{kj})\overset{d}{=}N^+(2^{k(j-1)},2^{kj})=2N(A^+(2^{k(j-1)},2^{kj}),-2^{kj},-2^{k(j-1)},2^{k(j-1)},2^{kj}).\]
Then Theorem \ref{thm1} and Proposition \ref{prop6} imply that
\begin{equation}\label{eqTlimit1}
\lim_{j\rightarrow \infty}ET^+(2^{k(j-1)},2^{kj})  \text{ exists and the limit only depends on } 1/2^k,
\end{equation}

\begin{equation}\label{eqTlimit2} 
\lim_{k\rightarrow\infty}\frac{\lim_{j\rightarrow \infty}ET^+(2^{k(j-1)},2^{kj})}{\log{2^k}}=\frac{\sqrt{3}}{2\pi}.
\end{equation}
The convergence of Ces\`{a}ro mean and \eqref{eqTlimit1} give for each $k\in\mathbb{N}$
\begin{equation}\label{eqCas}
\lim_{m\rightarrow\infty}\frac{\sum_{j=1}^m ET^+(2^{k(j-1)},2^{kj})}{m}=\lim_{j\rightarrow \infty}ET^+(2^{k(j-1)},2^{kj}).
\end{equation}
Therefore,
\begin{equation}\label{eqdlimit}
\lim_{k\rightarrow\infty}\lim_{n\rightarrow\infty}\frac{\sum_{j=1}^{\lfloor\log_{2^k}n\rfloor}ET^+(2^{k(j-1)},2^{kj})}{\log n}=\lim_{k\rightarrow\infty}\frac{\lim_{j\rightarrow\infty}ET^+(2^{k(j-1)},2^{kj})}{\log 2^k}=\frac{\sqrt{3}}{2\pi},
\end{equation}
where the first equality follows from \eqref{eqCas} and the second follows from \eqref{eqTlimit2}. It is not hard to show (the detail is a bit tedious and we refer the reader to a similar argument in the proof of Proposition 3.6 in \cite{Yao16}) that
\begin{equation}\label{eqequi}
\lim_{k\rightarrow\infty}\lim_{n\rightarrow\infty}\frac{E c_n^+}{\sum_{j=1}^{\lfloor\log_{2^k}n\rfloor}ET^+(2^{k(j-1)},2^{kj})}=1.
\end{equation}
Combining \eqref{eqdlimit} and \eqref{eqequi}, we get
\begin{equation*}
\lim_{n\rightarrow\infty}\frac{E c_n^+}{\log n}=\lim_{k\rightarrow \infty}\lim_{n\rightarrow\infty}\frac{E c_n^+}{\log n}=\frac{\sqrt{3}}{2\pi}.
\end{equation*}
A similar argument as the proof of Lemma 2.5 in \cite{Yao14} implies
\[\lim_{n\rightarrow\infty}\frac{c_n^+}{E c_n^+}=1.\]
The last two displayed equations complete the proof of Proposition \ref{propslln}.

\section{Variance of $c_n^+$}\label{secvar}
The proof of the limit result for Var$(c_n^+)$ is essentially the
same as for Var$(c_n)$ in \cite{Yao16}.  For the convenience of the
reader, we give the idea in the following.  As in \cite{Yao16}, we
use a modified martingale method introduced in \cite{KZ97}.  We
first introduce some notations.
For $j\in \mathbb{N}\cup\{0\}$, define the half-annulus
\begin{equation*}
A^+(j):=A^+(2^j,2^{j+1}).
\end{equation*}
Furthermore, define
\begin{align*}
&m(j):=\inf\{k\geq j: A(k)\mbox{ contains a blue half-circuit
surrounding 0}\},\\
&\mathcal {C}_j:=\mbox{the innermost blue half-circuit surrounding 0
in }A^+(m(j)),\\
&\mathscr{F}_j:=\sigma\mbox{-field generated by }\{t(v):v\in
\overline{\mathcal {C}}_j\},
\end{align*}
where $\overline{\mathcal {C}}_j:=\{\mbox{the sites in the finite
component of }\{\textbf{\textrm{V}}\cap\overline{\mathbb{H}}\}\backslash
\mathcal {C}_j\}\cup \mathcal {C}_j$. Denote by $T^+(0,\mathcal{C}_{j})$ the first-passage time in $\overline{\mathcal{C}}_j$ from $0$ to $\mathcal{C}_j$. That is,
\[T^+(0,\mathcal{C}_{j}):=\inf\{T(\gamma): \gamma \in \overline{\mathcal{C}}_j \text{ and }\gamma \text{ starts at }0\text{ and ends at a site in }\mathcal{C}_j \}.\]
For all $j\in \mathbb{N}$,
denote by $\mathscr{F}_{-j}$ the trivial $\sigma$-field. For
$k,q\in\mathbb{N}$, write
\begin{equation*}
T^+(0,\mathcal {C}_{kq})-E(T^+(0,\mathcal
{C}_{kq}))=\sum_{j=0}^q\left(E(T^+(0,\mathcal
{C}_{kq})|\mathscr{F}_{kj})-E(T^+(0,\mathcal
{C}_{kq})|\mathscr{F}_{k(j-1)})\right):=\sum_{j=0}^q\Delta_{k,j}.
\end{equation*}
Then $\{\Delta_{k,j}\}_{0\leq j\leq q}$ is a
$\mathscr{F}_{kj}$-martingale increment sequence.  Hence,
\begin{equation*}
\mathrm{Var}(T^+(0,\mathcal {C}_{kq}))=\sum_{j=0}^qE(\Delta_{k,j}^2).
\end{equation*}
One can use this sum to estimate Var$[c_n^+]$ with $2^{kq}\leq n\leq
2^{k(q+1)}$.  The proof proceeds as follows.  First, essentially in
the same way one can prove the half-circuit version of Lemmas 3.7,
3.8, 3.9, 3.10 in \cite{Yao16}. Then, similarly to the proof of
Proposition 3.11 in \cite{Yao16}, combining these Lemmas,
Propositions \ref{propn} and \ref{propnn}, one gets
the following proposition.
\begin{proposition}\label{propvar}
\begin{align*}
\lim_{n\rightarrow\infty}\frac{\mathrm{Var}(c_n)}{\log
n}=\frac{2\sqrt{3}}{\pi}-\frac{9}{\pi^2}.
\end{align*}
\end{proposition}

\section{Proofs of Theorem \ref{thmmain}, Corollary \ref{cormain}, Propositions \ref{propbb} and \ref{propcyl}}\label{seccor}
In this section, we complete the proofs of all results stated in the introduction.
\begin{proof}[Proof of Theorem \ref{thmmain}]
Theorem \ref{thmmain} is just the combination of Propositions \ref{propslln} and \ref{propvar}.
\end{proof}
\begin{proof}[Proof of Corollary \ref{cormain}]
Recall the definitions of $\tilde{D}_k$ and $N(\epsilon)$ in Sections \ref{secconf} and \ref{secfirst}. We can define $\tilde{D}_k^{\alpha}$  similarly by running a chordal SLE$_6$ in $\overline{\mathbb{D}^{\alpha}}$ from $e^{i\alpha}$ to $0$. Let
\[N^{\alpha}(\epsilon):=\sup\{k:\overline{\mathbb{D}}_{\epsilon}\subseteq\tilde{D}_k^{\alpha}\}.\]
Note that $\phi(z)=z^{\pi/\alpha}$ is a conformal bijection from $\mathbb{D}^{\alpha}$ to $\mathbb{D}^+$. Using the conformal invariance of chordal SLE, we have
\[N^{\alpha}(\epsilon)\overset{d}{=}N(\epsilon^{\pi/\alpha}).\]
So Proposition \ref{propn} implies
\[\lim_{\epsilon\downarrow 0}\frac{N^{\alpha}(\epsilon)}{\log(1/\epsilon)}=\frac{\sqrt{3}}{2\alpha} a.s.,~\lim_{\epsilon\downarrow 0}\frac{\mathbb{E}N^{\alpha}(\epsilon)}{\log(1/\epsilon)}=\frac{\sqrt{3}}{2\alpha},~\lim_{\epsilon\downarrow0}\frac{\mathrm{Var}(N^{\alpha}(\epsilon))}{\log(1/\epsilon)}=\frac{2\sqrt{3}}{\alpha}-\frac{9}{\pi\alpha}.\]
The rest of the proof is similar to that of Theorem \ref{thmmain}.
\end{proof}

\begin{figure}
\begin{center}
\includegraphics[width=0.3\textwidth]{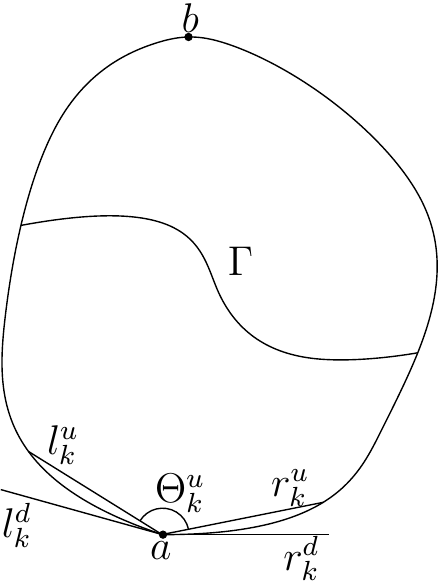}
\caption{An illustration of $l_k^d, l_k^u, r_k^d, r_k^u$, $\Theta_k^u$ and $\Gamma$.}\label{fig3}
\end{center}
\end{figure}

\begin{proof}[Proof of Proposition \ref{propbb}]
Note that the limit (in mean and thus in probability) of $T_{D_{\delta}}(a_{\delta}, b_{\delta})$ as $\delta\downarrow 0$ only depends on the local geometry near $a$ and $b$; this is because the expected number of $1$-clusters crossing in the bulk is finite by Theorem \ref{thm1}. One can approximate the portion of $\partial D$ near $a$ by sequences of secant lines and/or tangent lines. More precisely, suppose $\beta(t_0)=a$ where $t_0>0$. Then one can pick $\{\epsilon_k^u, k\geq 1\}$, $\{\epsilon_k^d, k\geq 1\}$, $\{\eta_k^u, k\geq 1\}$, $\{\eta_k^d, k\geq 1\}$ where  $\epsilon_{\cdot}^{\cdot}, \eta_{\cdot}^{\cdot}\geq 0$ and $\epsilon_k^{\cdot}\downarrow 0, \eta_k^{\cdot}\downarrow0$ as $k\rightarrow\infty$ such that: if we denote by $l_k^u$ (resp., $l_k^d$) the secant line segment between $\beta(t_0)$ and $\beta(t_0-\eta_k^u)$ (resp., $\beta(t_0-\eta_k^d)$) (if $\eta_k^{\cdot}=0$ we set $l_k^{\cdot}$ to be the left tangent line at $a$), and by $r_k^u$ (resp., $r_k^d$) the secant line segment between $\beta(t_0)$ and $\beta(t_0+\epsilon_k^u)$ (resp., $\beta(t_0+\epsilon_k^d)$) (if $\epsilon_k^{\cdot}=0$ we set $r_k^{\cdot}$ to be the right tangent line at $a$), then when we replace the portion of $\partial D$ at $\beta[t_0-\eta_k^u, t_0+\epsilon_k^u]$ by $l_k^u\cup r_k^u$ the resulting new domain is smaller (could be equal) and when we replace the portion of $\partial D$ at $\beta[t_0-\eta_k^d,t_0+\epsilon_k^d]$ by $l_k^d\cup r_k^d$ the resulting new domain is larger (could be equal). The angle subtended by $l_k^u$ and $r_k^u$ (resp., $l_k^d$ and $r_k^d$) is denoted by $\Theta_k^u$ (resp., $\Theta_k^d$). Let $\Gamma\subset D$ be a simple curve such that $a, b\notin \Gamma$, $a$ and $b$ are boundary points in different connected components of $D\setminus \Gamma$. See Figure \ref{fig3} for an illustration. Let $\Gamma_{\delta}$ be the lattice approximation of $\Gamma$. Define $T_{D_{\delta}}(a_{\delta},\Gamma_{\delta})$ to be the first passage time in $D_{\delta}$ between $a_{\delta}$ and $\Gamma_{\delta}$. Define $T_{D_{\delta}}^u(a_{\delta},\Gamma_{\delta})_k$ (resp., $T_{D_{\delta}}^d(a_{\delta},\Gamma_{\delta})_k$) to be the first passage time in the discrete approximation of the domain with boundary $(\partial D\setminus \beta[t_0-\eta_k^u, t_0+\epsilon_k^u])\cup l_k^u\cup r_k^u$ (resp., $(\partial D\setminus \beta[t_0-\eta_k^d, t_0+\epsilon_k^d])\cup l_k^d\cup r_k^d$) between $a_{\delta}$ and $\Gamma_{\delta}$. Then clearly
\[T_{D_{\delta}}^d(a_{\delta},\Gamma_{\delta})_k\leq T_{D_{\delta}}(a_{\delta},\Gamma_{\delta})\leq T_{D_{\delta}}^u(a_{\delta},\Gamma_{\delta})_k.\]
Corollary \ref{cormain} implies that
\begin{eqnarray*}
\lim_{\delta\downarrow 0}\frac{T_{D_{\delta}}^u(a_{\delta},\Gamma_{\delta})_k}{-\log\delta}=\frac{\sqrt{3}}{2\Theta_k^u}~a.s.,~ &&\lim_{\delta\downarrow 0}\frac{T_{D_{\delta}}^d(a_{\delta},\Gamma_{\delta})_k}{-\log\delta}=\frac{\sqrt{3}}{2\Theta_k^d}~a.s.,\\ \lim_{\delta\downarrow 0}\frac{E_{\delta}T_{D_{\delta}}^u(a_{\delta},\Gamma_{\delta})_k}{-\log\delta}=\frac{\sqrt{3}}{2\Theta_k^u},
&&\lim_{\delta\downarrow 0}\frac{E_{\delta}T_{D_{\delta}}^d(a_{\delta},\Gamma_{\delta})_k}{-\log\delta}=\frac{\sqrt{3}}{2\Theta_k^d}.
\end{eqnarray*}
It is clear that $\Theta_k^u\rightarrow \Theta_D(a)$ and $\Theta_k^d\rightarrow\Theta_D(a)$ as $k\rightarrow\infty$. Hence
\[\lim_{\delta\downarrow 0}\frac{T_{D_{\delta}}(a_{\delta},\Gamma_{\delta})}{-\log\delta}=\frac{\sqrt{3}}{2\Theta_D(a)}~a.s.,~\lim_{\delta\downarrow 0}\frac{E_{\delta}T_{D_{\delta}}(a_{\delta},\Gamma_{\delta})}{-\log\delta}=\frac{\sqrt{3}}{2\Theta_D(a)}.\]
Similar limits hold for $T_{D_{\delta}}(b_{\delta},\Gamma_{\delta})$. A standard argument as in (2.84) of \cite{KZ97} should yield
\[\lim_{\delta\downarrow 0}\frac{T_{D_{\delta}}(a_{\delta},b_{\delta})-T_{D_{\delta}}(a_{\delta},\Gamma_{\delta})-T_{D_{\delta}}(b_{\delta},\Gamma_{\delta})}{-\log\delta}=0 \text{ in probability}.\]
It is not hard to show the fractional expression of the last equation viewed as a sequence of $\delta$ is uniformly integrable. This completes the proof.
\end{proof}

\begin{proof}[Proof of Proposition \ref{propcyl}]
The basic idea is to show that both $E|c_n^+-s_{0,n}|$ and $E|c_n^+-s_{0,n}|^2$ are bounded above by some constant. Since the proof is very similar to the proof of Theorem 1.1 in \cite{Yao16}, we omit the details here.
\end{proof}

\section*{Acknowledgments}
We would like to thank the anonymous referees for many valuable comments and suggestions. We thank Greg Lawler and Chuck Newman for several useful discussions. The research of J.J. was partially supported by STCSM grant 17YF1413300 and that of C.-L.Y. by the National Natural Science Foundation of China (No. 11601505 and No. 11688101) and the Key Laboratory of Random Complex Structures and Data Science, CAS (No. 2008DP173182).


\begin{thebibliography}{99}
\bibitem{ADH16}
\textsc{A. Auffinger}, \textsc{M. Damron} and \textsc{J. Hanson}
(2017). \textit{50 years of first passage percolation}. University Lecture Series Vol. 68, American Mathematical Society.

\bibitem{BN11}
\textsc{V. Beffara} and \textsc{P. Nolin}
(2011). On monochromatic arm exponents for 2D critical percolation. \textit{Ann. Probab.} \textbf{39} 1286-1304.

\bibitem{BR06}
\textsc{B. Bollob\'{a}s} and \textsc{O. Riordan}
(2006). The critical probability for random Voronoi percolation in the plane is $1/2$. \textit{Probab. Theory Relat. Fields} \textbf{136} 417-468.

\bibitem{CN06}
\textsc{F. Camia} and \textsc{C. Newman}
(2006). Two-dimensional critical percolation: the full scaling limit. \textit{Comm. Math. Phys.} \textbf{268} 1-38.

\bibitem{CN07}
\textsc{F. Camia} and \textsc{C. Newman}
(2007). Critical percolation exploration path and SLE$_6$: a proof of convergence. \textit{Probab. Theory Relat. Fields} \textbf{139} 473-519.

\bibitem{Cur15}
\textsc{N. Curien}
(2015). A Glimpse of the conformal structure of random plannar maps. \textit{Commun. Math. Phys.} \textbf{333} 1417-1463.

\bibitem{DLW16}
\textsc{M. Damron}, \textsc{W.-K. Lam} and \text{W. Wang}
(2017). Asymptotics for 2D critical first passage percolation. \textit{Ann. Probab.} \textbf{45} 2941-2970.

\bibitem{Gri99}
\textsc{G. Grimmett}
(1999). \textit{Percolation}. Grundlehren der Mathematischen Wissenschaften [Fundamental Principles of Mathematical Sciences] \textbf{321}, Springer-Verlag, Berlin.


\bibitem{HW65}
\textsc{J. Hammersley} and \textsc{D. Welsh}
(1965). First-passage percolation, subadditive processes, stochastic networks, and generalized renewal theory. \textit{Proc. Internat. Res. Semin., Statist. Lab., Univ. California, Berkeley, Calif.} Springer-Verlag, New York, 61-110.

\bibitem{HS11}
\textsc{C. Hongler} and \textsc{S. Smirnov}
(2011). Critical percolation: the expected number of clusters in a rectangle. \textit{Probab. Theory Relat. Fields} \textbf{151} 735-756.

\bibitem{Kes86}
\textsc{H. Kesten}
(1986). Aspects of first-passage percolation. \textit{Lecture Notes in Math.} \textbf{1180}, Springer-Verlag, Berlin and New York, 125-264.

\bibitem{KZ97}
\textsc{H. Kesten} and \textsc{Y. Zhang}
(1997). A central limit theorem for ``crticial'' first-passage percolation in two-dimensions. \textit{Probab. Theory Relat. Fields} \textbf{107} 137-160.

\bibitem{Law05}
\textsc{G. Lawler}
(2005). \textit{Conformally Invariant Processes in the Plane}. Mathematical Surveys and Monographs Vol. 114, American Mathematical Society.

\bibitem{Law15}
\textsc{G. Lawler}
(2015). Minkowski content of the intersection of a Schramm-Loewner evolution (SLE) curve with the real line. \textit{J. Math. Soc. Japan} \textbf{67}(4), 1631-1669.

\bibitem{MM99}
\textsc{H. McKean} and \textsc{V. Moll}
(1999). \textit{Elliptic curves: function theory, geometry, arithmetic}. Cambridge University Press.

\bibitem{RS05}
\textsc{S. Rohde} and \textsc{O. Schramm}
(2005). Basic properties of SLE. \textit{Ann. Math.} \textbf{161} 883-924.


\bibitem{SSW09}
\textsc{O. Schramm}, \textsc{S. Sheffield} and \textsc{D.B. Wilson}
(2009). Conformal radii for conformal loop ensembles. \textit{Commun. Math. Phys.} \textbf{288}, 43-53.

\bibitem{She09}
\textsc{S. Sheffield}
(2009). Exploration trees and conformal loop ensembles.
\textit{Duke Math. J.} \textbf{147}, 79--129.

\bibitem{Smi01}
\textsc{S. Smirnov}
(2001). Critical percolation in the plane: conformal invariance, Cardy's formula, scaling limits. \textit{C. R. Acad. Sci. Paris Ser. I Math.} \textbf{333} 239-244.

\bibitem{SW78}
\textsc{R.T. Symthe} and \textsc{J.C. Wierman}
(1978). \textit{First-passage percolation on the square lattice}. Lecture notes in Mathematics,Vol. 671, Spring-Verlag.

\bibitem{Yao14}
\textsc{C.-L. Yao}
(2014). Law of large numbers for critical first-passage percolation on the triangular lattice. \textit{Electron. Communi. Probab.} \textbf{19}(18) 1-14.

\bibitem{Yao16}
\textsc{C.-L. Yao}
(2018). Limit theorems for critical first-passage percolation on the triangular lattice. \textit{Stochastic Process. Appl.} \textbf{128} 445-460.



\end{thebibliography}
\end{document}